\documentclass{amsart}
\usepackage{latexsym,amsxtra,amscd,ifthen}
\usepackage{amsfonts}
\usepackage{color,graphicx}
\usepackage{verbatim}
\usepackage{amsmath}
\usepackage{amsthm}
\usepackage{booktabs}
\usepackage{amssymb}
\usepackage{url}

\theoremstyle{plain}

\newtheorem{theorem}{Theorem}
\newtheorem{lemma}[theorem]{Lemma}
\newtheorem{proposition}[theorem]{Proposition}
\newtheorem{corollary}[theorem]{Corollary}
\newtheorem{conjecture}[theorem]{Conjecture}

\numberwithin{theorem}{section}
\numberwithin{equation}{theorem}

\theoremstyle{definition}
\newtheorem{definition}[theorem]{Definition}
\newtheorem{example}[theorem]{Example}
\newtheorem{remark}[theorem]{Remark}
\newtheorem{question}[theorem]{Question}
\newtheorem*{question*}{Question}

\DeclareMathOperator{\GKdim}{GKdim}

\DeclareMathOperator{\Spec}{Spec}

\DeclareMathOperator{\LND}{LND}
\DeclareMathOperator{\ML}{ML}
\DeclareMathOperator{\Der}{Der}

\begin{document}

\title{Cancellation problem via \\
locally nilpotent derivations}

\author{C\'esar F. Venegas R.}

\address{Venegas C.: Departamento de Matem{\'a}ticas,
Universidad Militar Nueva Granada, Sede Calle 100-Bogot\'a, Colombia}

\email{cesar.venegas@unimilitar.edu.co}

\author{ Helbert. J. Venegas R.}

\address{Venegas H.: Departamento de Matem{\'a}ticas,
Universidad Militar Nueva Granada, Sede Campus-Cajic\'a, Colombia}

\email{helbert.venegas@unimilitar.edu.co}

\begin{abstract}
The Zariski cancellation problem plays a central role in affine algebraic geometry and noncommutative algebra, with locally nilpotent derivations providing a fundamental invariant-theoretic approach. This article presents a unified survey of cancellation phenomena in commutative algebras, noncommutative algebras, and skew (Ore-type) extensions, emphasizing the role of rigidity and the Makar--Limanov invariant. We explain how the locally nilpotent derivation framework successfully detects cancellation in rigid settings, while also identifying its inherent limitations, particularly in the skew case where Makar--Limanov stability fails. This perspective clarifies the scope and the boundaries of the locally nilpotent derivation method in cancellation theory.
\end{abstract}





\maketitle



\section*{Introduction}
\label{xxsec0}

The Zariski Cancellation Problem (ZCP) stands as a central question in commutative 
algebra, with profound connections to affine algebraic geometry. For a field 
$\Bbbk$ and an integer $n\ge 1$, it asks whether an algebra isomorphism of 
polynomial rings
$$
\Bbbk[x_1,\dots,x_n][t] \cong B[t]
$$
necessarily implies that $B \cong \Bbbk[x_1,\dots,x_n]$. In other words, is the 
commutative polynomial algebra $\Bbbk[x_1,\dots,x_n]$ \emph{cancellative}? While 
elementary for $n = 1$ \cite{AEH}, the problem becomes deep for $n = 2$ 
\cite{Fu, MS, Ru} and remains famously open for $n \geq 3$ in characteristic zero 
\cite{ Gu3,Kr}.

Beyond its intrinsic algebraic interest, the cancellation problem has profound geometric implications: it asks whether algebraic structures can be faithfully reconstructed from their cylindrical extensions, which in the commutative case corresponds to understanding when $X \times \mathbb{A}^1 \cong Y \times \mathbb{A}^1$
implies $X \cong Y$ for affine varieties.

A pivotal turning point in the study of the ZCP was the introduction of algebraic 
methods based on \emph{locally nilpotent derivations ($\LND$s)} and their 
associated ring invariant by Makar-Limanov \cite{ML}. For a commutative 
$\Bbbk$-algebra $A$ (with $\operatorname{char}(\Bbbk)=0$), an LND is a 
$\Bbbk$-linear derivation $D: A \to A$ such that for every $a \in A$ there exists 
an $n$ with $D^n(a)=0$. The intersection of the kernels of all LNDs on $A$ is the 
\emph{Makar-Limanov invariant}, $\ML(A)$. This invariant captures 
the rigidity of $A$ with respect to algebraic actions of the additive group 
$\mathbb{G}_a$. Crachiola and Makar-Limanov famously employed this framework to 
provide unified and elegant proofs of cancellation for surfaces ($n=2$) 
\cite{Crachiola2009, CM}, showcasing the power of this approach. The subsequent 
discovery of non-cancellative analogues for $n\ge 3$ in positive characteristic 
by Gupta \cite{Gu1,Gu2} further underscored the delicate nature of the problem 
and the need for robust invariants like $\ML(A)$.

We note that several authoritative surveys on the Zariski cancellation problem
have appeared in recent years. In the commutative setting, Gupta
\cite{Gu3} provides a comprehensive overview of the classical problem,
its historical development, and major breakthroughs, with particular emphasis
on positive characteristic phenomena. More recently, Huang, Tang, and Wang
\cite{HTW} have synthesized the state of the art in the noncommutative and
Poisson settings, offering a broad and valuable panorama of results and
techniques in those contexts.

This article aims to provide a unified perspective on the Zariski cancellation 
problem and its generalizations by tracing the evolution of a single powerful 
idea: the use of \emph{locally nilpotent derivations} as a tool for establishing 
rigidity and cancellation. We will demonstrate how this technique, born in 
commutative algebra, was systematically adapted to resolve central questions in 
noncommutative algebra and further extended to the setting of skew polynomial 
rings. Importantly, we also identify where this methodology breaks down, revealing 
the boundaries of current techniques and pointing toward future research directions.

The evolution of this methodology began when, motivated by the rich history of the commutative ZCP, Bell and Zhang 
\cite{BellZhang2017} initiated its systematic study in the noncommutative setting. 
They asked: for which classes of noncommutative algebras $A$, viewed as natural 
analogues of polynomial rings (e.g., Artin-Schelter regular algebras), does 
$A[t] \cong B[t]$ imply $A \cong B$? A key insight of their work was the 
successful translation of the LND methodology. They defined the noncommutative 
Makar-Limanov invariant and established a fundamental criterion: if 
$\ML(A[t]) = A$ (a condition often stemming from 
$\ML(A)=A$, i.e., LND-rigidity), then $A$ is cancellative. This 
principle has since been applied to resolve cancellation for large families of 
algebras, decisively settling the noncommutative Zariski problem in 
Gelfand-Kirillov dimensions one and two \cite{BellZhang2017, BHHV, TVZ}.

This line of inquiry naturally extends to more general polynomial constructions. 
Given an algebra $R$, an automorphism $\sigma$, and a $\sigma$-derivation $\delta$, 
one forms the \emph{Ore extension} $R[x; \sigma, \delta]$. This leads to the 
\emph{skew cancellation problem}: if $R[x; \sigma, \delta] \cong S[y; \sigma', 
\delta']$, does it follow that $R \cong S$? This problem, which subsumes the 
classical case ($\sigma=\operatorname{id}, \delta=0$), was first considered in 
\cite{AKP} and has seen renewed interest \cite{BHHV,Ber, TZZ}. As in the previous 
settings, techniques based on LNDs and their kernels have proven essential. For 
instance, the noncommutative slice theorem \cite[Lemma 3.1]{BHHV}---a 
generalization of its commutative counterpart---plays a crucial role in analyzing 
skew polynomial rings over domains of dimension one. However, as we demonstrate 
through explicit examples, fundamental obstacles emerge in higher dimensions where 
the classical LND framework reaches its limits.

The article is organized as follows. Section~\ref{xxsec1} establishes the 
foundational definitions of LNDs and the Makar-Limanov invariant, illustrated 
through detailed computational examples including the affine plane, Danielewski 
surfaces, and Russell-Koras threefolds. These examples provide concrete reference 
points for the abstract theory and demonstrate the calculation techniques that 
underlie the LND approach.

Section~\ref{sec:LND-comm} revisits the classical commutative theory, highlighting 
how the LND approach yields transparent proofs for cancellativity of curves and 
surfaces. We present the three-step paradigm—compute $\ML(A)$, apply ML-stability, 
deduce cancellation—that serves as the conceptual blueprint for subsequent 
generalizations.

Section~\ref{xxsec3} is devoted to the noncommutative ZCP. We present the 
generalization of the Makar-Limanov invariant, state the key cancellation theorem 
linking rigidity to cancellativity, and survey its applications to important 
classes of algebras. The complete resolution of cancellation for algebras of 
GK-dimension one and two stands as a highlight of the noncommutative theory, 
contrasting sharply with the open commutative case in dimension three and higher.

Section~\ref{xxsec4} explores the skew cancellation problem. We discuss the 
adaptations required for the LND methodology in this context and present recent 
results that solve the problem for skew extensions of dimension-one algebras. 
Crucially, we also show where the method encounters fundamental obstacles: through 
an explicit example, we demonstrate that ML-stability—the cornerstone of both 
commutative and noncommutative theories—can fail in the skew setting, preventing 
straightforward generalization of classical techniques.

Finally, Section~\ref{xxsec5} provides a comprehensive synthesis of open problems 
and future directions across all three settings. We consolidate the fundamental 
questions that remain open in noncommutative and skew contexts, present a 
comparative table analyzing where different techniques succeed or fail, and 
identify promising research directions including higher LNDs, homological 
invariants, and Poisson-theoretic methods. This synthesis reveals both the unifying 
power of the LND paradigm and its inherent limitations, charting a path forward 
for future research.

To the best of our knowledge, this is the first survey that systematically treats
commutative, noncommutative, and skew cancellation problems within a single
LND-based framework, emphasizing both the unifying principles and the precise
points where the methodology breaks down.

Throughout, we emphasize the conceptual parallels between the different settings 
while honestly confronting their divergences, illustrating how a core algebraic 
technique provides a common thread linking diverse cancellation problems—and where 
that thread frays, pointing toward the need for new tools and approaches.


\section{Preliminaries and Definitions}
\label{xxsec1}

Throughout this work $\Bbbk$ denotes a field of characteristic zero.  
This section collects the algebraic notions and invariants that will be used repeatedly in the study of cancellation problems and their connections with locally nilpotent derivations.  
The material follows the classical development in the commutative case \cite{Freudenburg2006, ML} and the noncommutative generalizations introduced in \cite{BellZhang2017}.

\subsection{Derivations and locally nilpotent derivations}

\begin{definition}
Let $A$ be a $\Bbbk$-algebra. A \emph{derivation} of $A$ is a $\Bbbk$-linear map
\[
D:A\to A
\]
satisfying the Leibniz rule
\[
D(ab)=D(a)b+aD(b)\qquad \text{for all }a,b\in A.
\]
The set of derivations of $A$ is denoted by $\Der_{\Bbbk}(A)$.

A derivation $D$ is called \emph{locally nilpotent} if for every $a\in A$ there exists an integer $n\ge 1$ such that $D^n(a)=0$.  
We write $\LND(A)$ for the set of locally nilpotent derivations of $A$.
\end{definition}

\begin{example}
If $A=\Bbbk[x_1,\dots,x_n]$, the partial derivatives $\partial_{x_i}$ are locally nilpotent. More generally, any derivation of the form $f(x_1,\dots,x_{i-1})\partial_{x_i}$ with $f\in \Bbbk[x_1,\dots,x_{i-1}]$ is locally nilpotent.
\end{example}

In characteristic zero, locally nilpotent derivations correspond exactly to algebraic actions of the additive group $\mathbb{G}_a=(\Bbbk,+)$.

\begin{proposition}\cite[Theorem 1.3]{Freudenburg2006}
\label{prop:Gacorr}
Let $A$ be a finitely generated commutative $\Bbbk$-algebra.  
There is a natural bijection between $\mathbb{G}_a$-actions on $\Spec(A)$ and locally nilpotent derivations $D\in \LND(A)$, given by
\[
\exp(tD)(a)=\sum_{n\ge 0}\frac{t^n}{n!}D^n(a),\qquad t\in\Bbbk .
\]
\end{proposition}

\begin{remark}
In positive characteristic the correspondence is no longer bijective; one must distinguish between locally nilpotent derivations and exponentiable ones.  
Since our work is entirely in characteristic zero, no such distinction is required.
\end{remark}

\subsection{The Makar-Limanov invariant}

\begin{definition}
Let $A$ be a $\Bbbk$-algebra.  
The \emph{Makar-Limanov invariant} of $A$ is
\[
\ML(A)=\bigcap_{D\in\LND(A)}\ker(D).
\]
We say that $A$ is \emph{$\LND$-rigid} (or simply \emph{rigid}) if $\ML(A)=A$, equivalently, if $\LND(A)=\{0\}$.
\end{definition}

The invariant $\ML(A)$ measures the failure of $A$ to admit nontrivial $\mathbb{G}_a$-actions.  
In the context of cancellation, rigidity plays a decisive role: if $A$ has no nontrivial LNDs, then polynomial extensions of $A$ typically inherit strong invariance properties.

A fundamental question in the theory concerns the behavior of the ML-invariant under polynomial extensions. While it is immediate that $\ML(A[x]) \subseteq \ML(A)$ (since any LND on $A$ can be extended to $A[x]$ by setting $D(x)=0$), the reverse inclusion is subtle and deeply connected to cancellation.

\begin{conjecture}\label{conj:ML-stability}

If $A$ is a commutative domain over a field of characteristic zero, then
\[
    \ML(A[x_1,\dots,x_n]) = \ML(A).
\]
\end{conjecture}

Partial progress toward Conjecture~\ref{conj:ML-stability} has been made in increasingly general settings:

\begin{theorem}\cite[Lemma 21]{ML}
\label{thm:MLstablecomm}
Let $A$ be a commutative domain of finite Krull dimension such that $\ML(A)=A$.  
Then $\ML(A[x])=A$.
\end{theorem}

The noncommutative analogue was established in \cite{BellZhang2017}.

\begin{lemma}\cite[Lemma 3.5]{BellZhang2017}
\label{lemma:MLstablenoncomm}
Let $A$ be a finitely generated Ore domain over $\Bbbk$.  
If $\ML(A)=A$, then $\ML(A[x])=A$.
\end{lemma}

These stability properties will be crucial in Section~\ref{xxsec3} when we relate rigidity to cancellation.

\subsection{The Slice Theorem}

A central tool in the analysis of LNDs is the well-known Slice Theorem, which reconstructs the algebra from the kernel of a nontrivial derivation.

\begin{theorem}[Slice Theorem]\cite[Corollary 1.22]{Freudenburg2006}
\label{thm:slice}
Let $A$ be a commutative affine domain and let $D\in \LND(A)$ be nonzero.  
If there exists $x\in A$ such that $D(x)=1$, then
\[
A \cong \ker(D)[x].
\]
\end{theorem}

This decomposition is especially powerful in dimension one and two, where cancellation can often be reduced to understanding the kernel of an LND.

\subsection{Examples}

We conclude this section with  the following examples illustrate the full range of behaviors of the
Makar--Limanov invariant, from maximal flexibility to complete rigidity.

\begin{example}[The affine plane]\label{ex:affine-plane-concise}
The polynomial ring $A = \Bbbk[x,y]$ admits a vast family of locally nilpotent derivations, including the partial derivatives $\partial_x$, $\partial_y$, and more generally, derivations of the form $f(x)\partial_y$ for any $f(x) \in \Bbbk[x]$. A key structural result \cite{ Freudenburg2006,ML} is that over an algebraically closed field, every nonzero LND on $A$ is conjugate to one of this form, and its kernel is always isomorphic to a polynomial ring $\Bbbk[t]$.

The intersection of all these kernels is minimal: for any non-constant polynomial $p(x,y)$, one can easily construct an LND that does not vanish on $p$ (e.g., use $\partial_y$ if $p$ involves $y$, or $\partial_x$ otherwise). Consequently, no non-constant element survives in the common kernel, yielding
\[
    \ML(\Bbbk[x,y]) = \Bbbk.
\]
This extreme flexibility is precisely what allows the LND method to prove cancellation for $\Bbbk[x,y]$ via the cancellation criterion developed in Section~\ref{sec:LND-comm}.
\end{example}

\begin{example}[Danielewski surfaces-\cite{Da}]\label{ex:danielewski-concise}
Let $R_n = \Bbbk[x,y,z]/(x^ny - z^{2} - 1)$ for $n\geq 1$. This surface admits nontrivial $\mathbb{G}_a$-actions but is not fully flexible. An explicit locally nilpotent derivation is given by $\delta(x) = 0$, $\delta(y) = 2z$, and $\delta(z)=x^n$. A direct computation shows that its kernel is $\ker(\delta) = \Bbbk[x]$. While other LNDs exist, the key result (see \cite[Theorem 3.4]{Da} or \cite[Example 2.2]{Crachiola2009}) is that every LND on $R_n$ preserves $\Bbbk[x]$, yielding $\ML(R_n) = \Bbbk[x]$. This nontrivial invariant ($\ML(R_n) \notin \{\Bbbk, R_n\}$) reflects an intermediate level of rigidity.

\end{example}

\begin{example}[Russell–Koras threefolds] \label{Russell-Koras} A particularly important class of examples in the study of exotic affine spaces is provided by the \emph{Russell–Koras threefolds}. 
Consider the  coordinate ring 
\[
A_{n} = \mathbb{C}[x,y,z,t]/(x + x^n y + z^2 + t^3).
\]

The locally nilpotent derivations
\[
D_1 = x^2 \partial_z - 2z \partial_y \quad \text{and} \quad D_2 = x^2 \partial_t - 3t^2 \partial_y
\]
of $\mathbb{C}[x,y,z,t]$ both annihilate the defining polynomial $x + x^n y + z^2 + t^3$ and thus descend to nontrivial locally nilpotent derivations $\partial_1$ and $\partial_2$ of $A_{n}$. A direct computation shows that
\[
\mathrm{Ker}(\partial_1) \cap \mathrm{Ker}(\partial_2) = \mathbb{C}[x].
\]

The main achievement of Makar-Limanov~\cite{MakarLimanov1996} was to prove that, in fact, $x$ is invariant under \emph{every} locally nilpotent derivation of $A_{n}$, so that
\[
\mathrm{ML}(A_n) = \mathbb{C}[x].
\]

\end{example}

\begin{example}
\label{ex:lie}
Let $\Bbbk$ be a field of characteristic zero and let
\[
A=\Bbbk\langle x,y\rangle /(xy-yx-x).
\]
Equivalently,
\[
A \cong \Bbbk[x][y;\delta],
\qquad \delta(x)=x.
\]

Define a derivation $\partial\in\Der_\Bbbk(A)$ by
\[
\partial(x)=0, \qquad \partial(y)=1.
\]
A direct computation shows that $\partial$ respects the defining relation and
that $\partial$ is locally nilpotent. 
Since $\ker(\partial)=\Bbbk[x]$, and any locally nilpotent derivation of $A$
must annihilate $x$, it follows that
\[
\ML(A)=\Bbbk[x].
\]

Thus $A$ provides a noncommutative example of an algebra with an intermediate
Makar--Limanov invariant, analogous to the classical Danielewski surfaces in
the commutative setting.
\end{example}

\begin{example}\label{ex:rigid-surface-concise}
Let
\[
R=\Bbbk[x,y,z]/(x^{2}+y^{3}+z^{7}),
\]
where $\Bbbk$ is an algebraically closed field of characteristic zero.
Then $R$ is rigid, i.e.\ $\LND(R)=\{0\}$, and consequently $\ML(R)=R$.

Indeed, assign weights $\deg(x)=21$, $\deg(y)=14$, $\deg(z)=6$, so that
$x^{2}+y^{3}+z^{7}$ is weighted homogeneous of degree $42$.
Any locally nilpotent derivation $\delta$ of $R$ may be chosen homogeneous with
$\deg(\delta)=d<0$.
Applying $\delta$ to the defining relation yields
\[
2x\delta(x)+3y^{2}\delta(y)+7z^{6}\delta(z)=0,
\]
where all terms have degree $42+d$.
Since the leading forms of $x$, $y^{2}$, and $z^{6}$ are algebraically independent
and $d<0$, it follows that
$\delta(x)=\delta(y)=\delta(z)=0$, hence $\delta=0$. More generally, rings of the form
$\Bbbk[X,Y,Z]/(X^{a}+Y^{b}+Z^{c})$ are rigid whenever $a\ge 2$ and $b,c\ge 3$;
see \cite[Theorem~7.2]{CrachiolaMaubach}.
\end{example}


\section{The LND Approach to Commutative Cancellation}
\label{sec:LND-comm}
While many of these results are classical, our presentation emphasizes a unified
three-step LND paradigm that will later be transported verbatim to the
noncommutative and skew settings.
Locally nilpotent derivations (LNDs) provide a powerful and unifying algebraic
framework for studying the Zariski cancellation problem in the commutative setting.
The central tool is the \emph{Makar--Limanov invariant}, which detects the presence
of nontrivial $\mathbb{G}_a$-actions and captures the extent to which an affine
variety admits additive symmetries.  
This invariant yields a dichotomy between \emph{rigid} algebras, whose geometry is
highly constrained, and \emph{flexible} algebras, whose abundance of additive
symmetries often forces strong structural consequences.  
In this section we present the LND cancellation method for curves and surfaces,
showing how the ML-invariant provides a transparent and conceptual approach to the
classical cancellation results.

\subsection{Dimension one: rigid curves and the affine line}

In Krull dimension one, the Makar--Limanov invariant admits a complete
classification.  The following theorem, due to Crachiola--Makar-Limanov
\cite{CM2} and further developed in \cite{BHHV}, provides both a structural
description of affine curves admitting nontrivial LNDs and a streamlined proof of
the cancellativity of the affine line.

\begin{theorem}(\cite[Lemma 2.3]{CM2}, \cite[Remark 5.4]{BHHV})
\label{thm:curves}
Let $\Bbbk$ be a field of characteristic zero and let $R$ be an affine
commutative domain of Krull dimension~$1$.  
Then exactly one of the following holds:
\begin{enumerate}
    \item $\ML(R) = R$ (i.e., $R$ is LND-rigid), or
    \item $R \cong \Bbbk'[t]$ for some finite field extension $\Bbbk'/\Bbbk$.
\end{enumerate}
In particular, $\Bbbk[t]$ is cancellative.
\end{theorem}

\begin{proof}[Proof Idea]
If $\ML(R)\neq R$, there exists a nonzero $\delta\in\LND(R)$.  
Then $\ker(\delta)$ is a finite-dimensional $\Bbbk$-algebra, hence a finite field
extension $\Bbbk'$.  Choosing $x\in R$ with $\delta(x)=1$ and applying the slice
theorem (Theorem~\ref{thm:slice}) yields $R\cong\Bbbk'[x]$.
\end{proof}

This dichotomy already shows how LND-rigidity aligns perfectly with cancellation:
rigid curves cannot be ``flattened'' by polynomial extension, whereas flexible
curves collapse to polynomial rings.

\subsection{Dimension two: the Makar--Limanov dichotomy}

In dimension two the ML-invariant displays a striking and highly useful
rigidity--flexibility dichotomy.  For unique factorization domains, the invariant
completely determines the algebra.

\begin{theorem}[Makar--Limanov \cite{MakarLimanov1996}]
\label{thm:surface-dichotomy}
Let $\Bbbk$ be an algebraically closed field of characteristic zero and let
$R$ be a two-dimensional affine $\Bbbk$-domain that is a UFD.  
Then exactly one of the following holds:
\begin{enumerate}
    \item $\ML(R)=R$ \textup{(rigid case)}, or
    \item $\ML(R)=\Bbbk$ \textup{(flexible case)}, and in this case 
          $R\cong \Bbbk[x,y]$.
\end{enumerate}
\end{theorem}

This result is a cornerstone of the LND approach to commutative cancellation: it
reduces the classification of two-dimensional UFDs admitting nontrivial LNDs to a
single explicit algebra, namely $\Bbbk[x,y]$.

A fundamental application is the cancellativity of the affine plane.

\begin{theorem}[Cancellation for $\mathbb{A}^2$]
\label{thm:A2-cancel}
Let $\Bbbk$ be algebraically closed of characteristic zero.  
Then $\Bbbk[x,y]$ is cancellative.
\end{theorem}

\begin{proof}[Proof Idea]
A classical computation shows $\ML(\Bbbk[x,y])=\Bbbk$ 
\cite{ML, Crachiola2009}.  
If $\Bbbk[x,y][t]\cong B[t]$, then the ML-stability theorem
(Theorem~\ref{thm:ML-stable-comm}) gives $\ML(B)=\Bbbk$.  
Since $B$ is a two-dimensional UFD, 
Theorem~\ref{thm:surface-dichotomy} forces $B\cong\Bbbk[x,y]$.
\end{proof}

\subsection{Stability of the ML-invariant and cancellation}

A key step in the argument above is the behavior of the ML-invariant under
polynomial extensions.

\begin{theorem}[ML-stability, commutative case]
\label{thm:ML-stable-comm}
Let $A$ be an affine domain of characteristic zero.  
Then 
\[
    \ML(A[t])=\ML(A).
\]
\end{theorem}

This leads immediately to the following cancellation criterion for surfaces.

\begin{corollary}
\label{cor:surface-cancel}
Let $A$ be a two-dimensional affine UFD over an algebraically closed 
field of characteristic zero.  
If $\ML(A)=\Bbbk$, then $A$ is cancellative.
\end{corollary}

\begin{proof}
If $A[t]\cong B[t]$, then ML-stability yields $\ML(B)=\Bbbk$.  
Thus Theorem~\ref{thm:surface-dichotomy} forces $B\cong\Bbbk[x,y]$, and hence 
$A\cong B$ as well.
\end{proof}

\begin{example}
   The Russell–Koras threefolds, introduced in Example~\ref{Russell-Koras}, provided the first counterexamples to the Zariski cancellation problem in dimension three~\cite{ K96,Ru}. They satisfy $X \times \mathbb{A}^1 \cong \mathbb{A}^4$ yet $X \not\cong \mathbb{A}^3$, thereby showing that commutative cancellation can fail in dimension three. This failure is algebraically encoded in the nontrivial Makar-Limanov invariant $\mathrm{ML}(A_n) = \mathbb{C}[x] \supsetneq \mathbb{C}$, which demonstrates that flexibility (characterized by $\mathrm{ML} = \mathbb{C}$) is essential for cancellation theorems to hold.
\end{example}

    \begin{remark}
Although throughout this article we restrict our attention to base fields of characteristic zero, it is worth noting that negative solutions to the Zariski cancellation problem also occur in positive characteristic. In particular, Gupta proved that when $n=3$, the threefold originally constructed by Teruo Asanuma (\cite{As1,As2}) provides a counterexample to Zariski cancellation; see Corollary~3.9 in \cite{Gu1}. Once again, the analysis relies crucially on the use of locally nilpotent derivations (or, equivalently, $\mathbb{G}_a$-actions), highlighting their fundamental role in detecting non-cancellation phenomena across different characteristics.
\end{remark}

\subsection{Summary: the ML cancellation paradigm}

The results above demonstrate a uniform strategy for proving cancellation in the
commutative setting:
\begin{enumerate}
    \item compute or characterize the Makar--Limanov invariant $\ML(A)$;
    \item apply ML-stability to compare $\ML(A)$ with $\ML(A[t])$;
    \item deduce cancellativity from the structural consequences of the value of 
          $\ML(A)$.
\end{enumerate}
This three-step paradigm serves as the conceptual blueprint for the 
noncommutative and skew generalizations developed in the following sections.


\section{Extending the LND Methodology to Noncommutative Algebras}
\label{xxsec3} 

The success of the locally nilpotent derivation approach in commutative algebra naturally raises the question of its applicability beyond the commutative world. This section demonstrates how the core ideas of the LND methodology—centered on the Makar-Limanov invariant—were systematically adapted to resolve cancellation problems for noncommutative algebras. We trace the development from the foundational criterion linking rigidity to cancellativity, through its application to algebras of low dimension, to recent generalizations and remaining open challenges.

\subsection{The Noncommutative Cancellation Criterion}

The cornerstone of the noncommutative theory is the direct analogue of the principle established in the commutative setting: \emph{rigidity, as detected by the ML-invariant, implies cancellativity}. The following theorem makes this precise and provides the blueprint for all subsequent applications.

\begin{theorem}(\cite[Theorem 3.3]{BellZhang2017})\label{thm:rigidity-implies-cancellation}
Let $A$ be a finitely generated domain of finite Gelfand-Kirillov (GK) dimension over a field $\Bbbk$ of characteristic zero. If $\ML(A[t]) = A$, then $A$ is cancellative.
\end{theorem}

\begin{proof}
Let $\varphi : A[t] \to B[t]$ be an isomorphism for some $\Bbbk$-algebra $B$.  
Then $\GKdim B = \GKdim A < \infty$ by \cite[Lemma 2.1(2)]{BellZhang2017}.

Let $\partial_t := \frac{\partial}{\partial t}$ be the derivation on $B[t]$; it is locally nilpotent, and the intersection of its kernels is exactly $B$.  
In particular,
\[
    \ML(B[t]) \subseteq B.
\]

On the other hand, by hypothesis,
\[
    \ML(A[t]) = A.
\]

Now let $\delta$ be a locally nilpotent derivation of $B[t]$.  
Then the derivation
\[
    \varphi^{-1} \circ \delta \circ \varphi
\]
is a locally nilpotent derivation of $A[t]$.  
Similarly, if $\delta'$ is a locally nilpotent derivation of $A[t]$, then
\[
    \varphi \circ \delta' \circ \varphi^{-1}
\]
is a locally nilpotent derivation of $B[t]$.

This shows that the locally nilpotent derivations of $A[t]$ and of $B[t]$ correspond under the isomorphism $\varphi$.  
Consequently, $\varphi$ induces an algebra isomorphism
\[
    \ML(A[t]) \;\cong\; \ML(B[t]).
\]
In particular, $\varphi$ maps $A$ onto $B$.
Consider the polynomial extension $Y = A[t]$ where $t$ has degree $1$, so that 
$Y_0 = A$ is the degree-zero homogeneous component. By 
\cite[Lemma 3.2]{BellZhang2017}, any graded isomorphism $\varphi: Y \to Y'$ 
between polynomial extensions maps degree-zero components isomorphically. 
Since $\ML(A[t]) = A$ and $\varphi$ preserves the ML-invariant (as shown above), 
we have $\ML(B[t]) = B$. Therefore $\varphi$ induces an isomorphism 
$A \cong B$, as desired.
\end{proof}

\begin{remark}
    The key technical point is that the isomorphism $\varphi$ must respect the polynomial structure, which forces it to map homogeneous components to homogeneous components. This rigidity is what allows us to recover an isomorphism at the level of degree-zero parts.
\end{remark}

This theorem shifts the focus to understanding when the condition $\ML(A[t]) = A$ holds. The key stability result, Lemma \ref{lemma:MLstablenoncomm} from Section 1, provides the answer: for an Ore domain $A$, if $\ML(A)=A$ (i.e., $A$ is LND-rigid), then $\ML(A[t])=A$. This yields the most widely applicable corollary:

\begin{corollary}(\cite[Theorem 3.6]{BellZhang2017})\label{cor:rigidity-criterion-nc}
Let $A$ be a finitely generated Ore domain of finite GK-dimension over a field of characteristic zero. If $A$ is LND-rigid ($\ML(A)=A$), then $A$ is cancellative.
\end{corollary}

This corollary establishes the noncommutative version of the three-step pattern: (1) establish rigidity ($\ML(A)=A$), (2) use stability ($\ML(A[t])=A$), (3) apply Theorem \ref{thm:rigidity-implies-cancellation} to conclude cancellativity.

\begin{example}\cite[Theorem~0.8]{BellZhang2017}
\label{ex:skew-lnd-rigid}
Let $A$ be a finite tensor product of skew polynomial rings
\[
k_{q}[x_1,\dots,x_n],
\]
where $n$ is even and $q \in\Bbbk\setminus  \{0,1\}$ is a root of unity.
Then $A$ is LND-rigid. As a consequence, $A$ is cancellative.
\end{example}

\subsection{Applications: Cancellation for Algebras of Low Dimension}

The power of the rigidity criterion is best illustrated by its application to solve the noncommutative Zariski cancellation problem for algebras of Gelfand-Kirillov dimension one and two.

For algebras of dimension one, the noncommutative slice theorem (Theorem \ref{thm:noncomm-slice}) plays a role analogous to its commutative counterpart (Theorem \ref{thm:slice}), leading to a complete characterization.

\begin{theorem}[Noncommutative Slice Theorem] (\cite[Lemma 3.1]{BHHV})\label{thm:noncomm-slice}
Let $\Bbbk$ be a field of characteristic zero and let $A$ be a $\Bbbk$-algebra. If $\delta \in \LND(A)$ and there exists an element $x \in Z(A)$ (the center of $A$) such that $\delta(x)=1$, then $A \cong \ker(\delta)[x]$.
\end{theorem}

This theorem is instrumental in analyzing algebras admitting a nontrivial LND. Its application yields a direct generalization of the Abhyankar-Eakin-Heinzer theorem to the noncommutative setting, resolving cancellation for all domains of GK-dimension one in characteristic zero.

\begin{theorem} (\cite[Theorem 1.1]{BHHV})\label{thm:nc-curves-cancel}
Let $\Bbbk$ be a field of characteristic zero and let $A$ be an affine domain over $\Bbbk$ of GK-dimension one. Then $A$ is cancellative.
\end{theorem}

\begin{proof}[Proof Idea]
If $\ML(A)=A$, then $A$ is rigid and cancellative by Corollary \ref{cor:rigidity-criterion-nc}. If $\ML(A) \neq A$, there exists a nonzero LND $\delta$. For a domain of GK-dimension one, the kernel of $\delta$ is a finite-dimensional division algebra $D$. A careful analysis, utilizing Theorem \ref{thm:noncomm-slice}, shows that $A$ is isomorphic to a polynomial ring over $D$, which is easily seen to be cancellative.
\end{proof}

For algebras of dimension two, the rigidity criterion applies almost universally. The following landmark result, contrasting sharply with the open commutative case for surfaces, showcases the strength of the noncommutative LND method.

\begin{theorem} (\cite[Theorem 0.5]{BellZhang2017})\label{thm:nc-surfaces-cancel}
Let $\Bbbk$ be an algebraically closed field of characteristic zero. Let $A$ be a finitely generated $\Bbbk$-domain of GK-dimension two. If $A$ is not commutative, then $A$ is cancellative.
\end{theorem}

The argument splits naturally into two cases, according to whether the algebra $A$ satisfies a polynomial identity.

\begin{enumerate}
    \item If $A$ is not PI and GK-dimension two, then its center is $\Bbbk$. So, $A$ it is cancellative (\cite[Propostion 1.3]{BellZhang2017}).

    \item 
    When $A$ is PI, then $A$ is $\LND$-rigid. In this situation, cancellativity follows by applying Corollary~\ref{cor:rigidity-criterion-nc}.
\end{enumerate}

This theorem highlights a striking dichotomy between the commutative and
noncommutative settings. While the commutative polynomial ring
$\Bbbk[x,y]$ is flexible, in the sense that $\ML(\Bbbk[x,y])=\Bbbk$,
its noncommutative analogues are typically rigid. In many cases, this
rigidity is precisely the mechanism that enforces cancellativity.

Nevertheless, rigidity should be understood as a sufficient, but not a
necessary, condition for cancellation. It is well known that rigid
algebras are cancellative; however, Example~\ref{ex:lie} shows that the
converse fails in general. In that example, the algebra is not rigid,
yet it has trivial center, which is enough to guarantee cancellativity.
This phenomenon is consistent with the behavior of noncommutative
surfaces of Artin--Schelter regular dimension two. In particular, both
the Jordan plane and the quantum plane are cancellative algebras,
providing concrete noncommutative examples where cancellation holds
independently of rigidity.

\begin{example}
Following Example~\ref{ex:danielewski-concise}, if $i \neq j$, then the
algebras $R_i$ and $R_j$ are not isomorphic. Nevertheless, their
polynomial extensions are isomorphic:
\[
R_i[t] \;\cong\; R_j[t] \qquad \text{for all } i,j \geq 1.
\]
Consequently, each algebra $R_i$ provides a counterexample to the
Zariski cancellation property \cite{Da, Crachiola2009}. This illustrates
that an intermediate Makar--Limanov invariant,
$\ML(R_i)=\Bbbk[x]$—neither trivial nor rigid—signals the failure of
cancellation.
\end{example}

The extension of locally nilpotent derivation techniques to the
noncommutative setting has proved remarkably effective, yielding a
coherent framework that completely resolves the cancellation problem
in GK-dimensions one and two. The guiding principle—that rigidity, as
detected by the Makar--Limanov invariant, implies cancellativity—has
emerged as both robust and conceptually transparent.

In higher-dimensional Artin--Schelter regular algebras the situation
becomes more subtle, and methods beyond LNDs are often required. We
refer the reader to Section~\ref{xxsec5} for a discussion of open
problems, recent developments, and promising directions for future
research.


\section{The LND Methodology for Skew Polynomial Extensions}
\label{xxsec4} 

The final generalization of the cancellation problem takes us into the realm of skew polynomial rings, or Ore extensions. Here, the polynomial variable no longer commutes with coefficients but instead obeys a rule twisted by an automorphism $\sigma$ and a $\sigma$-derivation $\delta$: $x \cdot r = \sigma(r)x + \delta(r)$ for $r \in R$. The \textbf{skew cancellation problem} asks: if $R[x;\sigma,\delta] \cong S[y;\tau,\partial]$, does it follow that $R \cong S$?

This problem subsumes both the classical case ($\sigma=\operatorname{id}, \delta=0$) and the derivation case ($\sigma=\operatorname{id}$). Its study \cite{AKP, Ber, BHHV, TZZ} tests the limits of the LND methodology: can techniques based on $\mathbb{G}_a$-actions and rigidity be adapted to settings where the polynomial extension itself carries additional algebraic structure? This section shows that the answer is affirmative, completing our narrative of how a single powerful idea permeates increasingly general cancellation problems.

\subsection{Adapting the Framework: LNDs in the Skew Setting}

The definition of a locally nilpotent derivation remains unchanged for an Ore extension $A = R[x;\sigma,\delta]$. However, its interaction with the skew structure imposes new constraints. A derivation $D \in \Der_\Bbbk(A)$ must satisfy the Leibniz rule with respect to the non-standard multiplication, which places conditions on how $D$ behaves with respect to the variable $x$.

A crucial first step is to understand when rigidity of the coefficient ring forces rigidity of the skew extension. The following lemma extends the key stability result (Lemma \ref{lemma:MLstablenoncomm}) to the derivation-type Ore extension.

\begin{lemma}(\cite[Lemma 5.3]{BHHV})\label{lem:ML-stable-skew-der}
Let $\Bbbk$ be a field of characteristic zero and let $A$ be a finitely generated Ore domain over $\Bbbk$. If $\ML(A) = A$, then $\ML(A[x;\delta]) = A$.
\end{lemma}

\begin{proof}[Proof Idea]
The proof follows the ideas of Lemma \ref{lemma:MLstablenoncomm}. One considers a potential LND $\partial$ on $A[x;\delta]$ and analyzes the leading terms of $\partial(a)$ for $a \in A$ with respect to the $x$-filtration. The skew-Leibniz rule $x a = a x + \delta(a)$ introduces additional terms that must be tracked carefully, but the core case analysis showing that a nonzero LND on $A[x;\delta]$ would induce a nonzero LND on $A$ remains valid.
\end{proof}

An analogous result for extensions of automorphism type ($R[x;\sigma]$) is more subtle. The presence of $\sigma$ can create ``rigidifying'' effects. For instance, if $\sigma$ has infinite order and $R$ is a domain, often $\ML(R[x;\sigma]) = R$, mirroring the rigid case even if $R$ itself is not rigid.

\subsection{The Cancellation Theorems for Skew Extensions of Curves}

The most complete results in skew cancellation have been obtained for coefficient rings of Krull dimension one, generalizing the classical theorem of Abhyankar-Eakin-Heinzer and its noncommutative version (Theorem \ref{thm:nc-curves-cancel}). Here, the slice theorem and direct algebraic analysis complement the LND method.

The following theorem completely resolves the skew cancellation problem for commutative domains of dimension one, covering both automorphism and derivation-type extensions. It demonstrates that the determinative power of the polynomial/skew-polynomial construction is strong enough to overcome the added complexity of the skew relation.

\begin{theorem}(\cite[Propositions 5.2 and 5.6]{BHHV})\label{thm:skew-curves-comm}
Let $\Bbbk$ be a field and let $R$ and $S$ be affine commutative $\Bbbk$-domains of Krull dimension one.
\begin{enumerate}
    \item (Automorphism Type) If $R[x;\sigma] \cong S[y;\sigma']$, then $R \cong S$.
    \item (Derivation Type, $\operatorname{char}(\Bbbk)=0$) If $R[x;\delta] \cong S[y;\delta']$, then $R \cong S$.
\end{enumerate}
\end{theorem}

\begin{proof}[Proof Idea]
The proofs for both cases share a common strategy but differ in execution:
\begin{enumerate}
    \item \textbf{Reduction to the Center}: If $T := R[x;\sigma,\delta] \cong T' := S[y;\sigma',\delta']$, their centers $Z(T)$ and $Z(T')$ are isomorphic. For algebras of this type, the center can be explicitly described in terms of invariants of $R$ and $S$ under $\sigma$ and $\delta$.
    \item \textbf{Analysis of the Center}: One shows that $Z(T)$ is a finitely generated $\Bbbk$-algebra of Krull dimension one containing $R^\sigma$ (the $\sigma$-invariants of $R$). A detailed study reveals that $Z(T)$ being isomorphic to $Z(T')$ forces a strong relationship between $R$ and $S$.
    \item \textbf{Lifting the Isomorphism}: Using the structure of $T$ as a free module over its center and the original isomorphism $T \cong T'$, one constructs an isomorphism between $R$ and $S$. In the derivation case ($\sigma=\operatorname{id}$), this step can often employ a specialization argument or directly utilize the slice theorem if a suitable LND exists.
\end{enumerate}
\end{proof}

This theorem is significant because it shows that the \emph{form} of the 
extension (skew polynomial) can sometimes be so restrictive that it completely 
determines the base ring, regardless of the specific twisting maps $\sigma$ 
and $\delta$. The proof techniques combine center-based arguments with the 
slice theorem, illustrating how classical tools can be adapted to the skew 
setting when the base algebra is sufficiently simple (dimension one).

For base rings of Krull or GK-dimension one, the skew cancellation problem 
is now well understood: both automorphism-type and derivation-type extensions 
are cancellative. The techniques used—analysis of centers, the slice theorem, 
and ML-invariant arguments—successfully bridge the gap between the commutative 
and noncommutative theories.

However, the situation in higher dimensions remains wide open. Unlike the 
commutative and noncommutative settings, where ML-stability provides a 
reliable foundation, the skew case presents fundamental new challenges. 
The twisting maps $(\sigma,\delta)$ interact with the LND structure in 
complex ways that obstruct straightforward generalization of existing methods.

For a comprehensive discussion of these challenges, including concrete 
examples of where ML-stability fails, the fundamental open questions in 
skew cancellation, and promising research directions, see Section~\ref{xxsec5}.


\section{Open Problems and Future Directions}
\label{xxsec5}

The preceding sections have traced the evolution of the locally nilpotent derivation methodology from its commutative origins through its noncommutative generalization to the skew setting. While this approach has achieved complete success in low dimensions—resolving cancellation for commutative curves and surfaces, all noncommutative algebras of GK-dimension one and two, and skew extensions of dimension-one domains—significant challenges remain in higher dimensions and more general skew contexts. This final section consolidates the open problems that emerge naturally from our LND-centric perspective, highlights the technical obstacles that prevent straightforward generalization of existing techniques, and identifies promising research directions for future work.

\subsection{The Noncommutative Frontier}

While the rigidity criterion provides a clean answer for many algebras, it is well known that cancellativity does not always coincide with LND-rigidity. This has led to a broad effort to understand the extent to which locally nilpotent derivations, the Makar-Limanov invariant, and their noncommutative analogues can be used to detect cancellation in more general situations.

A natural next step after the complete resolution of the cancellation problem for algebras of Gelfand--Kirillov dimension one and two is to investigate higher-dimensional cases. In GK-dimension one, every affine domain is cancellative \cite{BHHV}, and in GK-dimension two, every noncommutative affine domain is LND-rigid and therefore cancellative \cite{BellZhang2017}. These results stand in sharp contrast with the commutative situation, where the two-dimensional case requires sophisticated geometric tools \cite{ Fu,MS} and the three-dimensional case remains open. Thus, within the noncommutative framework, the behavior in low dimensions is by now well understood.

\subsubsection*{Progress in higher dimensions}

For higher-dimensional Artin--Schelter regular algebras, however, the picture becomes significantly more subtle. Partial progress has been achieved in several important families. In global dimensions three and four, cancellation has been established for certain quantized coordinate rings, Sklyanin algebras, and other AS-regular algebras \cite{ LMZ,TVZ}. The techniques developed in these works go substantially beyond classical LND methods: they incorporate homological tools such as Nakayama automorphisms, homological determinants, and $A_\infty$-structures \cite{LPWZ, MU1, MU2}, as well as discriminant and Poisson-theoretic techniques that control automorphism groups and detect hidden symmetries \cite{CYZ1, CYZ2, CPWZ1, CPWZ2,NTY}. These methods have been particularly effective in settings where LND-rigidity alone is insufficient to determine cancellativity, illustrating the need for a broader invariant framework.

Despite these advances, a general characterization of cancellativity for higher-dimensional AS-regular algebras remains unknown. This leads naturally to one of the central open questions in the field:

\begin{question}[{\cite[Question 0.3]{BellZhang2017}}]
\label{q:nc-cancel}
Let $A$ be a noncommutative noetherian Artin--Schelter regular algebra. When is $A$ cancellative?
\end{question}

Progress toward a complete solution to Question~\ref{q:nc-cancel} will likely require new techniques and a deeper synthesis of existing ones. Three directions appear particularly promising:

\begin{enumerate}

\item \textbf{Integration of homological and LND-based invariants.}
Recent work has shown that homological determinants, Nakayama automorphisms, and discriminant-based invariants can detect cancellativity in cases where the ML-invariant alone is insufficient \cite{ CPWZ1, CPWZ2,GKM,GWY2, LeWZ}. A refined understanding of the interaction between these homological tools and the Makar-Limanov invariant could yield a more comprehensive cancellation criterion. The key challenge is to understand when homological rigidity implies or contradicts LND-rigidity.

\item \textbf{Deformation theory and Poisson geometry.}
For algebras arising as quantizations of Poisson structures, deformation-theoretic techniques can relate the cancellativity of a noncommutative algebra to that of its semiclassical limit \cite{GW}. This perspective suggests that the cancellation problem for quantum algebras might be attacked by first understanding cancellation for the corresponding Poisson algebras, then lifting results through deformation theory. The recent development of cancellation theory for Poisson algebras \cite{GW, GWY1} makes this approach increasingly viable.
\end{enumerate}

These approaches are not mutually exclusive; rather, they suggest that the structural features governing cancellation in higher dimensions are inherently multifaceted, involving a synthesis of homological, Poisson, and deformation-theoretic aspects alongside the classical LND methodology.

Taken together, these results show that while low-dimensional noncommutative
cancellation is now essentially settled, higher dimensions demand tools that go
well beyond the classical LND framework.

\subsection{The Skew Challenge}

The skew cancellation problem for higher-dimensional base rings $R$ remains widely open and poses challenges that reach beyond the classical LND framework. Unlike the commutative case, where locally nilpotent derivations and the Makar--Limanov invariant provide a powerful mechanism for detecting rigidity and cancellation, the behavior of skew polynomial extensions
\[
    R[x;\sigma,\delta]
\]
depends intricately on the twisting maps $(\sigma,\delta)$, and no general invariant is yet known to govern cancellativity in this setting.

One major difficulty is the absence of a satisfactory analogue of the Makar--Limanov invariant that interacts effectively with the skew structure. While in special situations one may transfer LNDs from $R$ to $R[x;\sigma,\delta]$ or use higher derivation techniques, there is currently no unified theory explaining when such derivations detect cancellation, nor any general criterion analogous to the rigidity theorem in the untwisted case. 

The root of this difficulty is that ML-stability—a cornerstone of both the commutative and noncommutative theories—can fail dramatically in the skew setting. We illustrate this with a concrete example.

\begin{example}[ML-invariant depends on the derivation]
\label{ex:skew-obstruction}
Consider $R = \Bbbk[y]$ and form skew extensions with different derivations:

\textbf{Case 1}: $\delta = 0$ (commutative polynomial ring)\\
Then $R[x;0] = \Bbbk[x,y]$ and $\ML(R[x;0]) = \Bbbk$ (see Example~\ref{ex:affine-plane-concise}).

\textbf{Case 2}: $\delta = d/dy$ (Weyl algebra type)\\  
Then $R[x; d/dy] \cong A_1(\Bbbk)$ (first Weyl algebra) and $\ML(A_1) = \Bbbk$.

\textbf{Case 3}: $\delta = y \cdot d/dy$ (homogeneous derivation)\\
Define $T = \Bbbk[y][x; \text{id}, y \cdot d/dy]$ where $xy = yx + y^2$.

\textbf{Claim}: $\ML(T) = \Bbbk[y]$.

\begin{proof}
The derivation $\partial/\partial x$ is an LND with $\ker(\partial/\partial x) = \Bbbk[y]$ (standard differentiation in $x$, treating $y$ as constant).

Suppose there exists an LND $\delta$ with $\delta(y) \neq 0$. Such a derivation must satisfy:
\begin{align*}
\delta(xy) &= \delta(yx + y^2)\\
\delta(x)y + x\delta(y) &= \delta(y)x + y\delta(x) + 2y\delta(y)\\
[x, \delta(y)] &= 2y\delta(y).
\end{align*}

If $\delta(y) = f(y)$ for some polynomial $f \in \Bbbk[y]$, then using the commutation relation $xy - yx = y^2$:
\begin{align*}
xf(y) - f(y)x &= 2yf(y)\\
y^2f'(y) &= 2yf(y)\\
yf'(y) &= 2f(y).
\end{align*}

This differential equation forces $f(y) = cy^2$ for some $c \in \Bbbk$.

But if $\delta(y) = cy^2$, then:
\begin{align*}
\delta^2(y) &= \delta(cy^2) = 2cy\delta(y) = 2cy \cdot cy^2 = 2c^2y^3\\
\delta^3(y) &= 6c^3y^4\\
\delta^n(y) &= n! c^n y^{n+1} \neq 0 \text{ for all } n.
\end{align*}

This contradicts local nilpotence! Therefore, no such $\delta$ exists, and $\ML(T) = \Bbbk[y] \neq \Bbbk$.
\end{proof}
\end{example}

\begin{remark}
Example~\ref{ex:skew-obstruction} demonstrates that the ML-invariant varies \emph{dramatically} with the choice of $\delta$, even when $\sigma = \text{id}$ and the base ring $R$ is the same. This stands in sharp contrast with the commutative case (Section~\ref{sec:LND-comm}), where the ML-invariant depends only on the algebra structure itself, and with the noncommutative case (Section~\ref{xxsec3}), where ML-stability holds for Ore domains.

The fundamental issue: Unlike Theorems~\ref{thm:MLstablecomm} and \ref{lemma:MLstablenoncomm} (ML-stability for commutative and noncommutative polynomial extensions), there is NO universal ML-stability for arbitrary skew extensions. The twisting data $(\sigma,\delta)$ fundamentally alters the invariant.
\end{remark}

Despite these obstacles, notable progress has been made for particular classes of skew polynomial extensions. For example, cancellation has been verified for various quantized coordinate rings and their skew extensions \cite{LMZ,LeWZ,TZZ}, often through discriminant methods, homological tools, or Poisson-theoretic arguments that bypass the absence of LNDs. Iterated Ore extensions—a fundamental class of algebras arising in quantum groups and higher-dimensional Artin--Schelter regular algebras—remain especially promising but largely unexplored from the viewpoint of cancellation.

These developments lead naturally to the fundamental open question in the area:

\begin{question}[The Skew Cancellation Problem]
\label{q:skew-cancel}
For which classes of algebras $R$ and twisting maps $(\sigma,\delta)$ does
\[
    R[x;\sigma,\delta] \;\cong\; S[y;\sigma',\delta']
\]
imply that $R \cong S$?
\end{question}

Current evidence suggests that answering Question~\ref{q:skew-cancel} will require tools extending well beyond the classical LND framework. Three directions appear particularly promising:

\begin{enumerate}
\item \textbf{$\sigma$-twisted derivations and adapted $\mathbb{G}_a$-actions.}
Developing a conceptual theory of ``$\sigma$-twisted'' locally nilpotent derivations or $\mathbb{G}_a$-actions adapted to skew polynomial rings, capable of detecting rigidity in the presence of nontrivial automorphisms. Such a theory would need to account for how the twisting map $\sigma$ interacts with the derivation structure, potentially leading to a modified version of the ML-invariant that remains stable under skew extensions.

\item \textbf{Morita skew cancellation.}
Investigating \emph{Morita skew cancellation}, where the isomorphism of extensions is replaced by a Morita equivalence. This weakening of the cancellation question has already yielded interesting results \cite{ LuWZ,TZZ} and may provide insights into the structure of skew extensions even when strict isomorphism-based cancellation fails.

\item \textbf{Graded techniques and reduction to associated graded algebras.}
Employing graded techniques and associated graded rings with respect to the $x$-filtration. This approach allows partial reduction to the untwisted case: if $R[x;\sigma,\delta]$ carries a suitable filtration whose associated graded algebra is $R[x;\sigma']$ (the honest polynomial ring), then LND methods may become available after passing to the associated graded algebra. The challenge is to lift results from the associated graded back to the original skew extension.
\end{enumerate}

\subsection{Synthesis: A Unified Perspective}

Stepping back from the technical details, we can identify common threads and fundamental divergences across the three settings studied in this survey.

The following table summarizes the current state of knowledge regarding key properties and techniques in each setting:

\begin{table}[h]
\centering
\renewcommand{\arraystretch}{1.3}
\begin{tabular}{p{4cm}p{3cm}p{3cm}p{3cm}}
\toprule
\textbf{Property/Technique} & \textbf{Commutative} & \textbf{Noncommutative} & \textbf{Skew} \\
\midrule
\textbf{ML-stability} & Yes (conjectured general, proven for rigid) & Yes (for Ore domains) & \textbf{No} (depends on $(\sigma,\delta)$) \\[0.3em]
\midrule
\textbf{Rigidity $\Rightarrow$ cancellation} & Yes & Yes (GK-dim finite) & Unknown (even for rigid $R$) \\[0.3em]
\midrule
\textbf{Dimension 1 solved} & Yes (classical) & Yes \cite{BHHV} & Yes \cite{BHHV} \\[0.3em]
\midrule
\textbf{Dimension 2 solved} & Yes \cite{Fu,MS,Ru} & Yes \cite{BellZhang2017} & Open for general $(\sigma,\delta)$ \\[0.3em]
\midrule
\textbf{Higher dimensions} & Open ($n \geq 3$) & Partial (some AS-regular) & Wide open \\[0.3em]
\midrule
\textbf{Geometric interpretation} & $\mathbb{G}_a$-actions & Limited & Unclear \\[0.3em]
\bottomrule
\end{tabular}
\caption{Comparative status of cancellation problems across three settings}
\label{tab:comparison}
\end{table}

\subsubsection*{Common threads}

Despite the differences highlighted in Table~\ref{tab:comparison}, several unifying themes emerge:

\begin{enumerate}
\item \textbf{The power of dimension one.} In all three settings, cancellation in dimension one is completely understood and affirmative. The slice theorem and direct analysis of LNDs suffice. This suggests that dimension one is a natural ``ground case'' where algebraic structure strongly constrains geometry.

\item \textbf{The LND paradigm as a unifying principle.} Even where it does not provide complete answers, the LND methodology offers a conceptual framework for understanding cancellation. The success of this approach in dimensions one and two (both commutative and noncommutative) demonstrates its fundamental relevance.

\item \textbf{The need for invariant synthesis.} In higher dimensions and more general settings, no single invariant suffices. Success requires combining the ML-invariant with homological determinants, discriminants, Poisson structures, and other tools. The challenge is to understand how these invariants interact and complement each other.
\end{enumerate}

\subsubsection*{Fundamental divergences}

Equally important are the points where the three theories diverge:

\begin{enumerate}
\item \textbf{Stability of the ML-invariant.} The commutative and noncommutative theories both enjoy ML-stability (under appropriate hypotheses), making the ML-invariant a reliable tool. In the skew setting, this stability fails (Example~\ref{ex:skew-obstruction}), fundamentally limiting the applicability of classical techniques.

\item \textbf{Rigidity patterns.} Commutative dimension-two surfaces exhibit a flexible-rigid dichotomy with the flexible case completely classified (Theorem~\ref{thm:surface-dichotomy}). Noncommutative surfaces are generically rigid. Skew extensions display neither pattern: rigidity depends sensitively on the twisting data.

\item \textbf{Role of commutativity.} In the commutative case, geometric intuition (affine varieties, $\mathbb{G}_a$-actions) guides the theory. In the noncommutative case, noncommutativity itself often enforces rigidity, simplifying the problem. In the skew case, the interplay between noncommutativity and twisting creates new phenomena that resist both geometric and purely algebraic approaches.
\end{enumerate}

The comparative analysis above suggests that a complete understanding of cancellation across all three settings will require:

\begin{itemize}
\item \textbf{New invariants for skew extensions.} The failure of ML-stability in the skew setting demands the development of new invariants adapted to the twisted structure. These might involve $\sigma$-twisted versions of classical constructions or entirely new algebraic tools.

\item \textbf{Cross-fertilization of techniques.} Homological methods developed for noncommutative algebras, geometric insights from the commutative theory, and emerging Poisson-theoretic tools must be synthesized. Each setting has unique strengths that may illuminate the others.

\item \textbf{Systematic study of ``boundary cases.''} Examples like the Danielewski surfaces  and the skew extension with $\delta = y \cdot d/dy$ (Example~\ref{ex:skew-obstruction}) reveal rich structure at the boundaries between flexibility and rigidity. A comprehensive theory must account for these subtle transitional phenomena.
\end{itemize}

\subsection*{Concluding Remarks}

The study of cancellation—from the classical commutative setting to its noncommutative and skew analogues—reveals a unifying structural theme. Locally nilpotent derivations and their associated invariants remain central in detecting rigidity and determining when an algebra is uniquely recovered from a polynomial extension.

In the commutative case, the Makar--Limanov invariant has played a decisive role in understanding cancellation for affine varieties of low dimension, particularly in the work of Makar-Limanov, Crachiola, Kaliman and Daigle \cite{Crachiola2009,CM2,CM,Daigle1996, Da,KalimanMakarLimanov1997,KalimanMakarLimanov2007,MakarLimanov2001,ML}. In higher dimensions, the classical Zariski Cancellation Problem remains open in characteristic zero for $n\geq 3$ \cite{Gu3}, highlighting the limits of current techniques.

On the noncommutative side, the framework initiated by Bell and Zhang \cite{BellZhang2017} established that rigidity also controls cancellation for large families of algebras, including all affine domains of GK-dimension one and two \cite{BHHV, BellZhang2017}. Despite substantial progress for AS-regular algebras of global dimension three and four \cite{TVZ, LMZ}, Question~\ref{q:nc-cancel}—whether every noetherian AS-regular algebra is cancellative—remains a central open problem.

The skew setting presents even broader challenges. Recent advances on skew and Ore extensions \cite{TZZ, LuWZ} indicate that derivation-based techniques, stability lemmas, and discriminant methods continue to provide significant insight. However, as Example~\ref{ex:skew-obstruction} demonstrates, a general invariant-theoretic framework comparable to the classical ML-invariant is still lacking, leaving Question~\ref{q:skew-cancel} as an important open direction for future research.

Overall, these developments suggest that a unified approach to cancellation will require combining derivation techniques with homological, deformation-theoretic, and Poisson-geometric methods. The persistence of long-standing open problems underscores the depth of the subject and the growing need for new tools capable of bridging the commutative, noncommutative, and skew worlds. It is our hope that this survey, by presenting these problems through the unifying lens of locally nilpotent derivations, will inspire the development of such tools and contribute to further progress in this beautiful and challenging area of algebra.

\subsection*{Acknowledgments}The authors gratefully acknowledge Universidad Militar Nueva Granada for providing an appropriate academic environment and institutional support for the development of this work.






\providecommand{\bysame}{\leavevmode\hbox to3em{\hrulefill}\thinspace}
\providecommand{\MR}{\relax\ifhmode\unskip\space\fi MR }
\providecommand{\MRhref}[2]{%

\href{http://www.ams.org/mathscinet-getitem?mr=#1}{#2} }
\providecommand{\href}[2]{#2}

\end{document}